\documentclass[11pt]{amsart}
\usepackage{mathtools, amssymb, amsfonts, amsthm, enumitem}
\usepackage{epsfig, psfrag, color, graphicx, tikz, bm, hyperref, pdfsync}
\usepackage{amsfonts}
\usepackage{epstopdf}

\newtheorem{thm}{Theorem}

\theoremstyle{plain}
\newtheorem{theo}[thm]{Theorem}
\newtheorem{lem}[thm]{Lemma}

\theoremstyle{definition}

\numberwithin{equation}{section}

\def\Cervonenkis{\v{C}ervonenkis}

\def\sq{\square}

\def\zz{\mathbb Z}

\def\rr{\mathbb R}

\def\al{\alpha}

\def\T{\mathbf{T}}

\def\sseq{\subseteq}

\def\wt{\widetilde}
\def\<{\langle}
\def\>{\rangle}

\def\Z{ {\text {\rm Z} } }

\def\0{{\mathbf 0}}

\def\.{\hskip.06cm}
\def\ts{\hskip.03cm}

\def\poly{\textup{\textsf{P}}}

\def\Z{\mathbb{Z}}
\def\R{\mathbb{R}}

\def\T{\mathcal{T}}

\newcommand{\cj}[1]{\overline{#1}}

\renewcommand{\b}{\cj{b}}

\newcommand{\x}{\mathbf{x}}

\renewcommand{\u}{\mathbf{u}}
\def\vv{\mathbf{v}}

\newcommand{\y}{\mathbf{y}}
\newcommand{\z}{\mathbf{z}}

\newcommand{\ex}{\exists\ts}
\renewcommand{\for}{\forall\ts}

\def\proj{\textup{proj}}

\def\NP{{\textup{\textsf{NP}}}}

\def\AP{\textup{AP}}

\def\VC{\textup{VC}}
\def\vc{\textup{vc}}
\def\pat{\bm}
\def\T{\mathcal{T}}

\def\fam{\mathcal{S}}
\def\lang{\mathcal{L}}
\def\model{\bm M}

\def\PA{\textup{PA}}
\def\SPA{\textup{Short-PA}}

\title{VC-dimension of short Presburger formulas}

\author[Danny Nguyen \and Igor Pak]{Danny Nguyen$^{\star}$ \and Igor~Pak$^{\star}$}

\thanks{\thinspace ${\hspace{-.45ex}}^\star$Department of Mathematics,
UCLA, Los Angeles, CA, 90095.
\hskip.06cm
Email:
\hskip.06cm
\texttt{\{ldnguyen,\ts{pak}\}@math.ucla.edu}}

\thanks{
\today}

\begin{document}
\maketitle

\begin{abstract}
We study VC-dimension of \emph{short formulas} in Presburger Arithmetic,
defined to have a bounded number of variables, quantifiers and atoms.
We give both lower and upper bounds, which are tight up to a polynomial
factor in the bit length of the formula.
\end{abstract}

\ \ {}

\vskip.7cm

\section{Introduction}

The notion of $\VC$-dimension was introduced by Vapnik and \v{C}ervonenkis in~\cite{VC}.
Although originally motivated by applications in probability and statistics,
it was quickly adapted to computer science, learning theory, combinatorics, logic and other areas.
We refer to~\cite{Vap} for the extensive review of the subject, and to~\cite{Che} for
an accessible introduction to combinatorial and logical aspects.

\subsection{Definitions of VC-dimension and VC-density}
Let $X$ be a set and $\fam \subseteq 2^{X}$ be a family of subsets in $X$.
A subset $A \subseteq X$ is \emph{shattered} by $\fam$ if for every subset
$B \subseteq A$, there is $S \in \fam$ with $B = S \cap A$.
In other words, every subset of $A$ can be cut out by $\fam$.
The largest size of a subset $A \subseteq X$ shattered by $\fam$ is called the
\emph{VC-dimension} of $\fam$, denoted by $\VC(\fam)$.
If no such largest size $|A|$ exists, we write $\VC(\fam)=\infty$.


The \emph{shatter function} is defined as follows:
$$
\pi_{\fam}(n) \. = \. \max\big\{|\fam \cap A| : A \sseq X, |A| = n\big\}\ts,
$$
where $|\fam \cap A|$ denotes the number of subsets $B \sseq A$ which can be cut out by $\fam$.
The \emph{VC-density} of $\fam$, denoted by $\vc(\fam)$ is defined as
$$
\inf \Big\{r \in \rr^{+} \; : \; \textup{limsup}_{n \to \infty} \frac{\pi_{\fam}(n)}{n^r} < \infty  \Big\}\ts.
$$
The classical theorem of Sauer and Shelah~\cite{Sauer,Shelah} states that
$$\vc(\fam) \. \le \. \VC(\fam)\ts.$$
In other words, $\pi_{\fam}(n) = O(n^{d})$ in case $\fam$ has finite VC-dimension~$d$.
In general, VC-density can be much smaller than VC-dimension, and also behaves a lot better
under various operations on~$\fam$.

\subsection{NIP theories and bounds on VC-dimension/density}
It is of interest to distinguish the first-order theories in which VC-dimension and VC-density behave nicely.
Let $\lang$ be a first-order language and $\model$ be an $\lang$-structure.
Consider a \emph{partitioned $\lang$-formula} $F(\x; \ts\y)$ whose free variables are separated into two groups $\x \in M^{m}$ (objects) and $\y \in M^{n}$ (parameters).
For each parameter tuple $\b \in M^{n}$, let
$$
S_{\b} \. = \.\big\{\x \in M^{m} : \model \models F(\x;\ts \b)\big\}.
$$
Associated to $F$ is the family $\fam_{F} = \big\{S_{\b} : \b \in M^{n}\big\}$.
We say that $F$ is NIP, short for ``$F$ does \emph{not} have the independence property'', if $\fam_{F}$ has finite VC-dimension.
The structure~$\model$ is called NIP if every partitioned $\lang$-formula $F$ is NIP in~$\model$.

One prominent example of an NIP structure is \emph{Presburger Arithmetic} $\PA = (\zz, < , +)$,
which is the first-order structure on $\zz$ with only addition and inequalities.
The main result of this paper are the lower and upper bounds on the VC-dimensions of $\PA$-formulas.
These are contrasted with the following notable bounds on the VC-density:

\begin{theo}[\cite{A+}]\label{th:A+}
Given a $\PA$-formula $F(\x;\ts \y)$ with $\y \in \zz^{n}$, $\vc(\fam_{F}) \le n$ holds.
\end{theo}

In other words, VC-density in the setting of PA depends only on the dimension of parameter variables~$\y$,
and thus completely independent of the object variables $\x$, let alone other quantified variables or
the description of~$F$.  This follows from a more general result in~\cite{A+}, which says that every
\emph{quasi-o-minimal} structure satisfies a similar bound on the VC-density.  We refer to~\cite{A+} for the
precise statement of this result and for the powerful techniques used to bound the VC-density.

Karpinski and Macintyre raised a natural question whether similar bounds would hold for the VC-dimension.
In~\cite{KM1}, they gave upper bounds for the VC-dimension in some \emph{o-minimal} structures (PA is not one),
which are polynomial in the parameter dimension $n$.  Later, they extended their
arguments in~\cite{KM2} to obtain upper bounds on the \emph{VC-density}, this time linear in $n$.
To our knowledge, no effective upper bounds on the VC-dimensions of general PA-formulas exist in the literature.\footnote{We were informed by Matthias Aschenbrenner that in~\cite{KM2}, the authors claimed to have an effective bound on the VC-dimensions of PA-formulas.
However, we cannot locate such an explicit bound in any papers.
}

\subsection{Main results}
We consider PA-formulas with a fixed number of variables (both quantified and free).
Clearly, this also restricts the number of quantifier alternations in $F$.
The atoms in $F$ are linear inequalities in these variables with some integer constants and coefficients (in binary).
Given such a formula~$F$, denote by $\phi(F)$ the length of $F$, i.e.,
the total bit length of all symbols, operations, integer coefficients and constants
in~$F$.


We can further restrict the form of a PA-formula by requiring that it does not contain too many inequalities.
For fixed $k$ and $t$, denote by $\SPA_{k,t}$ the family of PA-formulas with at most $k$ variables (both free and quantified) and $t$ inequalities.
When $k$ and $t$ are clear, a formula $F \in \SPA_{k,t}$ is simply called a \emph{short Presburger formula}.
In this case, $\phi(F)$ is essentially the total length of  a bounded number of integer coefficients and constants.
Our main result is a lower bound on the VC-dimension of short Presburger formulas:

\begin{theo}\label{th:lower-nd}
For every $d$, there is a short Presburger formula $F(x;y) = \ex \u \; \for \vv \; \Psi(x,y,\u,\vv)$ in the class $\SPA_{10,18}$
with
\begin{equation*}
\phi(F)\.  = \. O(d^2) \quad \, \text{and} \quad \, \VC(F) \ge d\ts.
\end{equation*}
Moreover, $\Psi$ can be computed in probabilistic polynomial time in~$d$.
Here $x,y$ are singletons and $\u \in \zz^{6},\vv \in \zz^{2}$.
\end{theo}

So in contrast with VC-density, the VC-dimension of a PA-formula $F$ crucially depends on the actual length $\phi(F)$.
For the formulas in the theorem, we have:
$$
\VC(F) \. = \. \Omega\bigl(\phi(F)^{1/2}\bigr)\., \quad \text{and} \quad \vc(F) \le 1,
$$
where the last inequality follows by Theorem~\ref{th:A+}.
Note that if one is allowed an unrestricted number of inequalities in~$F$,
a similar lower bound to Theorem~\ref{th:lower-nd} can be easily established by an elementary combinatorial argument.
However, since the formula $F$ is short, we can only work with a few integer coefficients and constants.


The construction in Theorem~\ref{th:lower-nd} uses a number-theoretic technique that employs continued fractions to encode a union of many arithmetic progressions.
This technique was explored earlier in~\cite{NP2} to show that various decision problems with short Presburger sentences are intractable.
In this construction we need to pick a prime roughly larger than $4^{d}$, which can be done in probabilistic polynomial time in $d$.
This can be modified to a deterministic algorithm with run-time polynomial in~$d$, at the cost of increasing $\phi(F)$:

\begin{theo}\label{th:lower}
For every $d$, there is a short Presburger formula $F(x;y) = \ex \u \; \for \vv \; \Psi(x,y,\u,\vv)$ in the class $\SPA_{10,18}$
with
$$
\phi(F)\.  = \. O(d^3) \quad \, \text{and} \quad \, \VC(F) \ge d\ts.
$$
Moreover, $\Psi$ can be computed in deterministic polynomial time in~$d$.
\end{theo}

We conclude with the following polynomial upper bound for the VC-dimension of all
(not necessarily short) Presburger formulas in a fixed number of variables:

\begin{theo}\label{th:upper}
For a Presburger formula $F(\x;\ts \y)$ with at most $k$ variables (both free and quantified), we have:
$$\VC(F) \. = \. O\bigl(\phi(F)^c\bigr),
$$
where $c$ and the $O(\cdot)$ constant depend only on~$k$.
\end{theo}

This upper bound implies that Theorem~\ref{th:lower-nd} is tight up to a polynomial factor.
The proof of Theorem~\ref{th:upper} uses an algorithm from~\cite{NP1} for decomposing a semilinear set, i.e., one defined by a PA-formula, into polynomially many simpler pieces.
Each such piece is a polyhedron intersecting a periodic set, whose VC-dimensions can be bounded by elementary arguments.

We note that the number of quantified variables is vital in Theorem~\ref{th:upper}.
In~$\S$\ref{fin-rem:general}, we construct PA-formulas $F(x;y)$ with $x,y$ singletons and many quantified variables, for which $\VC(F)$ grows doubly exponentially compared to $\phi(F)$.

\bigskip

\section{Proofs}
We start with Theorem~\ref{th:lower}, and then show how it can be modified to give Theorem~\ref{th:lower-nd}.
\subsection*{Proof of Theorem~\ref{th:lower}}\label{ss:lower_proof}
Let $A = \{1,2,\dots,d\}$ and $\fam = 2^{S}$.
Since $\fam$ contains all of the subsets of $A$, we have $\VC(\fam) = d$.
We order the sets in $\fam$ lexicographically.
In other words, for $S, S' \in \fam$, we have $S > S'$ if $\sum_{i \in S} 2^{i} > \sum_{i \in S'} 2^{i}$.
Thus, the sets in $\fam$ can be listed as $S_{0} > S_{1} > \dots > S_{2^{d}-1}$, where $S_{0} = A$ and $S_{2^{d}-1} = \varnothing$.
Next, define:
\begin{equation}\label{eq:T_def_1}
T \coloneqq \bigsqcup_{0 \le j < 2^{d}} \{i + dj \;:\; i \in S_{j}, 1 \le i \le d\}.
\end{equation}
We show in Lemma~\ref{lem:def_T} below that the set $T$ is definable by a short Presburger formula $G_{T}(t)$ with only $8$ quantified variables and $18$ inequalities.
Using this, it is clear that the parametrized formula
\begin{equation*}
F_{T}(x;y) \; \coloneqq \; G_{T}(x + dy)
\end{equation*}
describes the family $\fam$ (with $y$ as the parameter), and thus has $\VC$ dimension $d$.
We remark that $G_{T}$ has only $1$ quantifer alternation (see below).
\hfill$\sq$

\begin{lem}\label{lem:def_T}
The set $T$ is definable by a short Presburger formula $G_{T}(t) = \ex \u \; \for \vv \; \Psi(t,\u,\vv)$ with $\u \in \zz^{6}, \vv \in \zz^{2}$ and $\Psi$ a combination of $18$ inequalities with length $\phi(G_{T}) = O(d^3)$.
\end{lem}
\begin{proof}
Our strategy is to represent the set $T$ as a union of arithmetic progressions (APs).
In~\cite{NP2}, we already gave a method to define any union of APs by a short Presburger formula of polynomial size.
For each $1 \le i \le d$, let  $J_{i} = \{j \ts : \ts 0 \le j < 2^{d},\ts  i \in S_{j} \}$.
From~\eqref{eq:T_def_1}, we have:
\begin{equation}\label{eq:T_def_2}
T = \bigsqcup_{i=1}^{d} (i + dJ_{i}).
\end{equation}
From the lexicographic ordering of the sets $S_{j}$, we can easily describe each set $J_{i}$ as:
\begin{equation}\label{eq:j_def}
J_{i} = \{x + 2^{i}y \;:\; 0 \le x < 2^{i-1},\; 0 \le y < 2^{d-i}\}.
\end{equation}
So each set $J_{i}$ is not simply an AP, but the Minkowski sum of two APs.
However, we can easily modify each $J_{i}$ into an $\AP$ by defining:
\begin{equation}\label{eq:wt_j_def}
J'_{i} = \{2^{d}x + 2^{i}y \;:\; 0 \le x < 2^{i-1},\; 0 \le y < 2^{d-i}\}.
\end{equation}
It is clear that $J'_{i}$ is an AP that starts at $0$ with step size $2^{i}$, and ends at $2^{i+d-1} - 2^{i}$.
Let
\begin{equation}\label{eq:wt_T_def}
T' = \bigsqcup_{i=1}^{d} (i + d J'_{i}),
\end{equation}
which is a union of $d$ $\AP$s.
Using the construction from~\cite{NP2}, we can define $T'$ by a short Presburger formula:
\begin{equation*}
t' \in T' \quad \iff \quad \ex \z \quad \for \vv \quad \Phi\big(t',\z,\vv\big),
\end{equation*}
where $t' \in \zz,\; \z,\vv \in \zz^2$ and  $\Phi$ is a Boolean combination of at most $10$ inequalities.
This construction works by finding a single continued fraction $\al = [a_{0} \ts;\ts a_{1}, \dots, a_{4d-1}]$ whose successive convergents $p_{k}/q_{k} = [a_{0} \ts;\ts a_{1}, \dots, a_{k}]$ encode the starting and ending points of our $d$ APs.
We refer to Section~4 in~\cite{NP2} for the details.
Each term $a_{k}$ in that construction is at most the product of the largest terms in the $d$ APs we want to encode.
For each $i$, the largest term $m_{i} := i + d(2^{i+d-1} - 2^{i})$ in $i + d J'_{i}$ has length $O(d)$.
Thus, the product $\prod_{i=1}^{d} m_{i}$
has length $O(d^{2})$, and so does each term $a_{k}$.
Therefore, the final continued fraction $\al$ is a rational number $p/q$, with length $O(d^{3})$.
This implies that $\phi(\Phi) = O(d^3)$ as well.


To get a formula for $T$, note that from~\eqref{eq:T_def_2},~\eqref{eq:j_def},~\eqref{eq:wt_j_def} and~\eqref{eq:wt_T_def}, we have:
\begin{align*}
t \in T \quad \iff \quad \ex \. t', r,r',s \;\; : \quad
& t' \in T' ,\quad 1 \le r \le d ,\quad \wt  0 \le r' < 2^{d} ,
\\
& t' = r + d(2^{d}s + r') ,\quad t = r + d(s + r') \ts.\text{\footnotemark}
\end{align*}\footnotetext{Here each equality is a pair of inequalities.}
\hspace{-.4em}Replacing $t' \in T'$ by $\ex \z \; \for \vv \; \Phi(t',\z,\vv)$, we get a formula $G_{T}(t)$ defining $T$ with $8$ quantified variables $t',r,r',s \in \zz$, $\z,\vv \in \zz^{2}$ and $18$ inequalities.
Note that $t',r,r',s$ and $\z$ are all existential variables, so $G_{T}$ has the form $\ex \u \; \for \vv \; \Psi(t,\u,\vv)$ with $\u \in \zz^{6}$ and $\vv \in \zz^{2}$.
\end{proof}

\smallskip

\subsection*{Proof of Theorem~\ref{th:lower-nd}}\label{ss:lower_proof-nd}
Note that the construction of $F_{T}$ and $G_{T}$ in the proof above is deterministic with run-time polynomial in $d$,
again as a consequence of the construction in~\cite{NP2}.  Since in Theorem~\ref{th:lower-nd} we need only
the existence of a short Presburger formula with high VC-dimension, our lower bound can be
improved to $\VC(F) \ge c\sqrt{\phi(F)}$, for some $c>0$, as follows.
Recall that $m_{i} = i + d(2^{i+d-1} - 2^{i})$ is the largest element in the arithmetic progression $i + d J'_{i}$ in~\eqref{eq:wt_T_def}.
Pick the smallest prime $p$ larger than $\max(m_{1},\dots,m_{d}) \approx d4^{d}$.
This prime $p$ can substitute for the large number~$M$ in Section~4.1 of~\cite{NP2},
which was (deterministically) chosen as $1+\prod_{i=1}^{d}m_{i}$, so that it is larger and coprime to all $m_{i}$'s.
The rest of the construction follows through.
Note that $\log(p) = O(d)$ by Chebyshev's theorem, which implies that the final
continued fraction $\al = [a_{0} \ts;\ts a_{1}, \dots, a_{4d-1}]$ has length $O(d^{2})$.
This completes the proof. \hfill$\sq$

\smallskip

\subsection*{Proof of Theorem~\ref{th:upper}}\label{th:upper_proof}
Let $F(\x;\y)$ be a Presburger formula in $\x \in \zz^{m}$, $\y \in \zz^{n}$ with $n'$ other quantified variables, where $k=m+n+n'$ is fixed.
In~\cite{NP1} (Theorem~5.2), we gave the following polynomial decomposition on the semilinear set defined by $F$:
\begin{equation}\label{eq:semilinear}
  \Sigma_{F} \; \coloneqq \;  \big\{(\x,\y) \in \Z^{m+n} \; : \; F(\x;\y) = \text{true}\big\} \quad = \quad \bigsqcup_{j=1}^{r} \, R_{j} \cap \pat T_{j}.
\end{equation}
Here each~$R_{j}$ is a polyhedron in~$\R^{m+n}$, and each $\pat T_{j} \sseq \Z^{m+n}$ is a periodic set, i.e.,
a union of several cosets of some lattice $\T_{j} \subseteq \Z^{m+n}$.
In other words, the set defined by $F$ is a union of $r$ pieces, each of which is a polyhedron intersecting a periodic set.
Our decomposition is algorithmic, in the sense that the pieces $R_{j}$ and lattices $\T_{j}$ can be found in time $O\big(\phi(F)^{c}\big)$, with $c$ and $O(\cdot)$ depending only on $k$.
The algorithm describes each piece $R_{j}$ by a system of inequalities and each lattice $\T_{j}$ by a basis.
Denote by $\phi(R_{j})$ and $\phi(\T_{j})$ the total binary lengths of these systems and basis vectors, respectively.
These also satisfy:
\begin{equation}\label{eq:length_bound}
\sum_{j=1}^{r} \phi(R_{j}) + \phi(\T_{j}) = O\big(\phi(F)^{c}\big).
\end{equation}

Each $R_{j}$ can be written as the intersection $H_{j1} \cap \dots \cap H_{jf_{j}}$, where each $H_{jk}$ is a half-space in $\rr^{m+n}$, and $f_{j}$ is the number of facets of $R_{j}$.
Note that $f_{j} \le \phi(R_{j}) = O\big(\phi(F)^{c}\big)$.
We rewrite~\eqref{eq:semilinear} as:
\begin{equation}\label{eq:semilinear2}
\Sigma_{F} \quad = \quad \bigsqcup_{j=1}^{r} \, H_{j1} \cap \dots \cap H_{jf_{j}} \cap \pat T_{j}.
\end{equation}
Therefore, the set $\Sigma_{F}$ is a Boolean combination of $f_{1}+\dots+f_{r}$ half-spaces and $r$ periodic sets.
In total, there are
\begin{equation}\label{eq:num_basic_sets}
f_{1}+\dots+f_{r} + r = O\big(\phi(F)^{c}\big)
\end{equation}
of those basic sets.

For a set $\Gamma \subseteq \rr^{m+n}$ and $\y \in \zz^{n}$, denote by $\Gamma_{\y}$ the subset $\{\x \in \zz^{m} : (\x,\y) \in \Gamma\}$ and by $\fam_{\Gamma}$ the family $\{\Gamma_\y : \y \in \zz^{n}\}$.
For a half-space $H \subset \rr^{m+n}$, it is easy to see that $\VC(\fam_{H}) = 1$.
For each periodic set $\pat T_{j}$ with  period lattice $\T_{j}$, the family $\fam_{\pat T_{j}}$
has cardinality at most $\det(\T_{j} \cap \zz^{n}) \le  2^{O(\phi(\T_{j}))}$.
Thus, we have
\begin{equation}\label{eq:VC_Tj}
\VC(\fam_{\pat T_{j}}) \le \log | \fam_{\pat T_{j}} | = O\big(\phi(\T_{j})\big).
\end{equation}

Let $\Gamma_{1},\dots,\Gamma_{t} \subseteq \zz^{m + n}$ be any $t$ sets with $\VC(\fam_{\Gamma_{i}}) = d_{i}$.
By an application of the Sauer-Shelah lemma (\cite{Sauer,Shelah}), if $\Sigma$ is any Boolean combination of $\Gamma_{1},\dots,\Gamma_{t}$, then we can bound $\VC(\fam_{\Sigma})$ as:
\begin{equation*}
\VC(\fam_{\Sigma}) = O\big((d_{1}+\dots+d_{t}) \log (d_{1}+\dots+d_{t})\big).
\end{equation*}
Applying this to~\eqref{eq:semilinear2}, we get $\VC(\fam_{\Sigma_{F}}) = O(\ell \log \ell)$, where
\begin{equation*}
\; \ell \; = \;\sum_{j=1}^{r} \bigg( \VC(\fam_{\pat T_{j}}) + \sum_{j'=1}^{f_{j}} \VC(\fam_{H_{jj'}}) \bigg) \; \le \; \sum_{j=1}^{r} \VC(\fam_{\pat T_{j}}) + f_{j}.
\end{equation*}
By~\eqref{eq:length_bound},~\eqref{eq:num_basic_sets} and~\eqref{eq:VC_Tj}, we have $\ell=O\big(\phi(F)^{c}\big)$.
We conclude that $\VC(F) =  O\big(\phi(F)^{2c}\big)$.

\hfill$\sq$

\bigskip

\section{Final remarks and open problems}

\subsection{}
The proof of Theorem~\ref{th:lower-nd} is almost completely effective except
for finding a small prime~$p$ larger than a given integer~$N$.  This problem
is considered to be computationally very difficult in the deterministic case, and only exponential
algorithms are known (see~\cite{LO,TCH}).

\subsection{}
Our constructed short formula $F_T$ is of the form $\ex \for$.
It is interesting to see if similar polynomial lower bounds are obtainable
with existential short formulas.
For such a formula $F(\x;\y) \ts = \ts \ex \z \, \Phi(\x,\y,\z)$,
the expression $\Phi(\x,\y,\z)$ captures the set of integer points $\Gamma$
lying in a union of some polyhedra $P_{i}$'s.
Note that the total number of polyhedra and their facets should be bounded, since we are working with short formulas.
Therefore, $F$ simply capture the pairs  $(\x,\y)$ in the projection of $\Gamma$ along the $\z$ direction.
Denote this set by $\proj(\Gamma)$.
The work of Barvinok and Woods~\cite{BW} shows that $\proj(\Gamma)$ has a \emph{short generating function}, and can even be counted efficiently in polynomial time.
In our construction, the set that yields high VC-dimension is a union arithmetic progressions, which cannot be counted efficiently unless $\poly = \NP$ (see~\cite{MS}).
This difference indicates that $\proj(\Gamma)$ has a much simpler combinatorial structure, and may not attain a high VC-dimension.


\subsection{}\label{fin-rem:general}
One can ask about the VC-dimension of a general PA-formula with no restriction on the number of variables, quantifier alternations or atoms.
Fischer and Rabin famously showed in~\cite{FR} that PA has decision complexity at least doubly exponential in the general setting.
For every $\ell > 0$, they constructed a formula $\text{Prod}_{\ell}(a,b,c)$ of length $O(\ell)$ so that for every triple
$$0 \le a,b,c \, < \, 2^{2^{2^{\ell}}},$$
we have \ts $\text{Prod}_{\ell}(a,b,c) = \text{true}$ \ts if and only if \ts $ab = c$.
Using this ``partial multiplication'' relation, one can easily construct a formula
$F_{\ell}(x;y)$ of length $O(\ell)$ and VC-dimension at least $2^{2^{\ell}}$.
This can be done by constructing a set similar to $T$ in~\eqref{eq:T_def_1} with $d$
replaced by $2^{2^{\ell}}$ using $\text{Prod}_{\ell}$.
We leave the details to the reader.

Regarding upper bound, Oppen showed in~\cite{O} that any PA-formula $F$ of length $\ell$ is equivalent to a quantifier-free formula $G$ of length $2^{2^{2^{c \ell}}}$ for some universal constant $c>0$.
This implies that $\VC(G)$, and thus $\VC(F)$, is at most triply exponential in $\phi(F)$.
We conjecture that a doubly exponential upper bound on $\VC(F)$
holds in the general setting.
It is unlikely that such an upper bound could be established by straightforward quantifier elimination, which generally results in triply exponential blow up (see~\cite[Thm~3.1]{Wei}).

\bigskip

\subsection*{Acknowledgements}
We are grateful to Matthias Aschenbrenner and Art\"{e}m Chernikov 
for many interesting conversations and helpful remarks.  This paper was
finished while both authors were visitors at MSRI; we are thankful
for the hospitality, great work environment and its busy schedule. 
The second author was partially supported by the~NSF.

\end{document}